\documentclass[12pt]{article}
\usepackage[T2A]{fontenc}
\usepackage[utf8]{inputenc}
\usepackage{amssymb,amsmath,amsfonts,amsthm,amscd,latexsym,verbatim,graphics,epsfig,indentfirst,xcolor,array,ulem}
\usepackage{geometry}
\def\id{\mathrm{id}}

\def\Aut{\mathrm{Aut}}

\def\Inn{\mathrm{Inn}}
\def\Imm{\mathrm{Im}}

\def\Ss{\mathrm{S}}

\def\PSp{\mathrm{PSp}}

\def\SU{\mathrm{SU}}

\newtheorem{lemma}{Lemma}
\newtheorem{proposition}{Proposition}
\newtheorem{theorem}{Theorem}

\newtheorem{example}{Example}
\newtheorem{definition}{Definition}

\begin{document}

\sloppy

\hfill{20D06 (MSC2020)}

\begin{center}
{\Large On Rota--Baxter operators on finite simple groups of Lie type} 

\smallskip

Alexey Galt, Vsevolod Gubarev
\end{center}

\begin{abstract}
Rota--Baxter operators on groups were introduced by L. Guo, H.~Lang, Yu.~Sheng in 2020.
In 2023, V. Bardakov and the second author showed that all Rota--Baxter operators on simple sporadic groups are splitting, i.\,e. they correspond to exact factorizations of groups.
In 2024, the authors of the current paper described all non-splitting Rota--Baxter operators on alternating groups.

Now we describe Rota--Baxter operators on finite simple exceptional groups of Lie type and projective special linear groups of degree two.

\end{abstract}

\section{Introduction}

In 1960~\cite{Baxter}, G. Baxter introduced the notion of Rota--Baxter operator on an algebra.
After more than 60 years of the study, a lot of connections and applications of Rota--Baxter operators were found: with different versions of the Yang--Baxter equation~\cite{Aguiar00,BelaDrin82,Semenov83}, pre- and postalgebras~\cite{Aguiar00,BBGN2011,GuoMonograph}, double Poisson algebras~\cite{DoubleLie2,DoubleLie,DoublePoissonFree}, etc. 

The notion of Rota--Baxter operator (of weight~$\pm1$) on groups was suggested by L.~Guo, H. Lang, and Yu. Sheng in~\cite{Guo2020}.
This subject has generated considerable interest among specialists and is rapidly developing in various areas~\cite{BG,BG2,Catino2023,Das,RBHopf,GV2017}.

Rota---Baxter operators have also been studied on Lie groups since the work~\cite{Guo2020} under the assumption that an RB-operator must be smooth, see~\cite{Jiang,Jiang2}.
Recently, S.~Skresanov~\cite{Skresanov} proved that there are only trivial RB-operators on a compact simple Lie group.

Given a group~$G$, one has always trivial RB-operators $B_e\colon g \to e$ and $B_{-1}\colon g\to g^{-1}$. 
The analog of the projection operator for algebras is the following:
given an exact factorization $G = HL$, the map
$$
B\colon hl \to l^{-1}
$$
is a~Rota--Baxter operator on~$G$. We will call it a splitting RB-operator. It was shown in~\cite{BG} that a given Rota--Baxter operator $B$ on a group~$G$ is splitting if and only if $\Imm(B\widetilde{B}) = 1$, where 
$$
\widetilde{B} \colon G\to G, \quad \widetilde{B}(g) = g^{-1}B(g^{-1}) 
$$
is another RB-operator on~$G$ defined with the help of~$B$.
Different constructions of non-splitting Rota--Baxter operators on groups were suggested in~\cite{BG}, but there are no non-splitting RB-operators on all simple sporadic groups~\cite[Theorem~53]{BG}. Continuing the study of Rota--Baxter operators on finite simple groups, all RB-operators on the alternating groups $\operatorname{A_n}$ were described in~\cite{GG}. In particular, it was shown that there are non-splitting RB-operators $B$ on $\operatorname{A_n}$ if and only if $n$ satisfies certain conditions, and there is an infinite number of such ones.

In the current paper we continue the classification of RB-operators on finite simple groups. Our first result deals with exceptional groups. 

\begin{theorem}\label{th:exceptional}
    The finite simple exceptional groups of Lie type have only trivial Rota--Baxter operators. 
\end{theorem}

The remaining but ``largest'' class of finite simple groups is the class of classical groups. In~\cite{GG} a new kind of equivalence on a set of Rota--Baxter operators was introduced (see the definition below). Our second result classifies all Rota--Baxter operators on a simple linear group of dimension $2$ up to this equivalence.

\begin{theorem}\label{th:PSL2}
    The groups $\operatorname{PSL_2(q)}$ do not have non-splitting Rota--Baxter operators. 
    There are $s$ non-trivial non-equivalent splitting Rota--Baxter operators on $\operatorname{PSL_2(q)}$, where
  \begin{itemize}
      \item[$(a)$] $s=3$ for $q=11$;
      \item[$(b)$] $s=2$ for $q\in\{7,23,59\}$; 
      \item[$(c)$] $s=1$ for $q\not\equiv1\!\pmod4$, $q\notin\{7,11,23,59\}$;
      \item[$(d)$] $s=0$ for $q\equiv1\!\pmod4$.
  \end{itemize}
\end{theorem}

All factors $\Imm(B)$ and $\Imm(\widetilde{B})$ of $\operatorname{PSL_2(q)}$ are given in Table~\ref{table}.

Theorems~\ref{th:exceptional} and~\ref{th:PSL2} show that there are no non-splitting Rota--Baxter operators on these groups. It is natural to ask does it true for all simple groups of Lie type? The answer is negative, and as shown in Example~\ref{Ex:PSp_6(2)}, the group $\operatorname{PSp_6(2)}$ has a non-splitting Rota--Baxter operator. 

In Lemma~\ref{lem:R=2} we describe a construction of non-splitting Rota--Baxter operators, which includes non-splitting RB-operators on $\operatorname{PSp_6(2)}$ and on the alternating groups. 

All known Rota--Baxter operators on finite simple groups arise from exact factorizations of the corresponding group.  
In Lemma~\ref{Extension} we provide a new construction of a~non-splitting Rota--Baxter operator on a (not necessarily finite) group $G$. In this case $G$ is an extension of an abelian group by a cyclic group. 
We apply Lemma~\ref{Extension} to construct a~non-splitting Rota--Baxter operator on a group of order~16 in Example~\ref{ex:extraspecial_group}. 
This is the first RB-operator that does not correspond to an exact factorization or to a homomorphism into an abelian subgroup.

\section{Notation and Preliminary results}
Our notation is standard. Let $q$ denote a power of a prime number $p$. 
The cyclic group of order $n$ is denoted by $\mathbb{Z}_n$ or $n$, and the elementary abelian group of order $p^m$ by~$p^m$.

Given a normal subgroup $A$ and a subgroup $B$, an arbitrary extension and a split extension of $A$ by $B$ are denoted by $A.B$ and $A\!:\!B$, respectively. We write $x^y=y^{-1}xy$ and $[x,y]=x^{-1}x^y$. We write $H\leqslant G$ and $H\unlhd G$ if $H$ is a subgroup or $H$ is a normal subgroup of $G$, respectively.

We recall the definition and some properties of a Rota--Baxter operator on a group.

\begin{definition}[\cite{Guo2020}]
Let $G$ be a group.
A map $B\colon G\to G$ is called a Rota--Baxter operator of weight~1 if
\begin{equation}\label{RB}
B(g)B(h) = B( g B(g) h B(g)^{-1} )
\end{equation}
holds for all $g,h\in G$.
\end{definition}

Note that a notion of Rota--Baxter operator on an abelian~$G$ coincides with homomorphism.

\begin{proposition}[\cite{Guo2020}]\label{prop:initial}
Let $(G,B)$ be a Rota--Baxter group. Then

(a) $\Imm(B)$ and $\ker(B)$ are subgroups of~$G$,

(b) $B(e) = e$,

(c) $B(g)^{-1} = B(B(g)^{-1}g^{-1}B(g))$ for every $g \in G$,

(d) $B(h) = B(gh)$ for all $h\in G$ and $g\in \ker(B)$.
\end{proposition}

\begin{proposition}[\cite{Guo2020}]\label{prop:Btilde}
Let $(G,B)$ be a Rota--Baxter group. Then

(a) the map $\widetilde{B}$ defined as follows, $\widetilde{B}(g) = g^{-1}B(g^{-1})$ is a Rota--Baxter operator on~$G$,

(b) given $\varphi\in \Aut(G)$, $B^{(\varphi)} = \varphi^{-1}B\varphi$
is a~Rota--Baxter operator on $G$.
\end{proposition}

\begin{proposition}[\cite{Guo2020}]\label{prop:Derived}
Let $(G,\cdot,B)$ be a Rota--Baxter group. Then

(a) The pair $(G, \circ )$ with the product
\begin{equation}\label{R-product}
g\circ h = gB(g)hB(g)^{-1}, \quad g,h\in G,
\end{equation}
is a group.

(b) The operator $B$ is a Rota--Baxter operator on the group $(G,\circ)$.

(c) The map $B\colon (G,\circ) \to (G,\cdot)$
is a homomorphism of Rota--Baxter groups.
\end{proposition}

We will denote the group $(G, \circ )$ by $G^{(\circ)}$.

\begin{proposition}[\cite{Guo2020}]\label{Prop:STSh} 
If $(G, B)$ is an RB-group, then

(a) $\ker(B)$ and $\ker(\widetilde{B})$ are normal in $G^{(\circ)}$,

(b) $\ker(B) \unlhd \Imm(\widetilde{B})$ and $\ker(\widetilde{B}) \unlhd \Imm(B)$ in $G$,

(c) We have the isomorphism of the quotient groups
\begin{equation}\label{FactorIso}
\Imm(\widetilde{B})/\ker(B)\simeq \Imm(B)/\ker(\widetilde{B}),
\end{equation}

(d) We have the factorization
\begin{equation}\label{ImageFactorization}
G = \Imm(\widetilde{B})\Imm(B).
\end{equation}
\end{proposition}

\begin{proposition}\label{Prop:simple}
Let $B$ be a Rota--Baxter operator on a finite non-abelian simple group~$G$. 
    \begin{itemize}
        \item[{(1)}] If $G$ has no non-trivial factorizations, then $B$ is trivial; 
        \item[{(2)}] If $B$ is non-splitting, then $1<\ker(\widetilde{B})<\Imm(B)$ and $\Imm(B)$ is not simple. 
    \end{itemize}
\end{proposition}
\begin{proof}
    (1) \cite[Corollary~52]{BG}.\\
    (2) follows from the proof of \cite[Theorem~53]{BG}.
\end{proof}

\begin{proposition}[\cite{Jiang}]\label{prop:RBviaSubgroups}
Let $B$ be a map defined on a finite group~$G$.
Define $H_B = \{(B(g),gB(g))\mid g\in G\}\subset G\times G$.

(a) If $B$ is an RB-operator on $G$, then $H_B$ is a subgroup of $G\times G$.

(b) Let $H$ be a subgroup of $G\times G$.
Then there exists an RB-operator $B$ on $G$ such that $H = H_B$ if and only if 
$|H|=|G|$ and $\{ba^{-1}\mid (a,b)\in H\} = G$.    
\end{proposition}

Given an element $x$ of a group $G$, the automorphism $g\mapsto g^x$ is denoted by $\alpha_x$.
Let us consider the following subgroup of~$\Aut(G\times G)$:
$$
Q(G) = \{ (\varphi,\varphi\,\alpha_x)\mid \varphi\in\Aut(G),\,x\in G\} 
 \cup \{ \tau((\varphi,\varphi\,\alpha_x))\mid \varphi\in\Aut(G),\,x\in G \},
$$
where 
$\tau((\varphi,\varphi\,\alpha_x))\colon (g,h)\to ( \varphi\,\alpha_x(h), \varphi(g)  )$.

\begin{proposition}[{\cite[Proposition~11]{GG}}]\label{prop:equivalentRB}
Let $B$ be an RB-operator on a finite group~$G$.
For every $\Phi\in Q(G)$, there exists an RB-operator $B'$ on $G$ such that $\Phi(H_B) = H_{B'}$.
\end{proposition}

A pair of RB-operators $B,B'$ on a group $G$ is called {\it equivalent}, if 
$\Phi(H_B) = H_{B'}$ for some $\Phi\in Q(G)$.
Proposition~\ref{prop:Btilde} can be interpreted in terms of equivalent RB-operators as follows.
If $\Phi = \tau((\id,\id))$, we have $B' = \widetilde{B}$.
If $\Phi = (\varphi,\varphi)$, then $B' = B^{(\varphi)}$.

The next proposition shows that this equivalence preserves main properties of RB-operators.

\begin{proposition}[{\cite[Lemma~3]{GG}}]
Let~$G$ be a finite group and $B,B'$ be a pair of equivalent RB-operators on~$G$.
Then 

(a) $G^{(\circ)}(B)\simeq G^{(\circ)}(B')$,

(b) $B$ is trivial if and only if $B'$ is trivial,

(c) up to action of $B\to \widetilde{B}$, we have $\Imm(B)\simeq \Imm(B')$, $\Imm(\widetilde{B})\simeq \Imm(\widetilde{B'})$,

(d) $\Imm(B\widetilde{B})\simeq \Imm(B'\widetilde{B'})$,

(e) $B$ is splitting if and only if $B'$ is splitting,

(f) if $G = \ker(B)\cdot \Imm(B)$, then up to the action of $B\to \widetilde{B}$,
one has $G = \ker(B')\cdot \Imm(B')$.
\end{proposition}

\section{New constructions of Rota--Baxter operators}    

\begin{proposition}[{\cite[Proposition~5.1]{BG}}]
If $G$ is a group and $H$ is its abelian subgroup,
then any homomorphism (or antihomomorphism) $B\colon G \to H$ is a Rota---Baxter operator.
\end{proposition}

\begin{proposition}[{\cite[Corollary~1]{GG}}]\label{coro:old}
Given an exact factorization $G = HL$, let $C$ be a~Rota---Baxter
operator on $L$. If $\Imm(\widetilde{C})\leqslant L$ normalizes $H$, then a~map $B \colon G \to G$ defined by the formula $B(hl) = C(l)$, where $h \in H$ and $l \in L$, is a Rota---Baxter operator.
\end{proposition}

\begin{lemma}[{\cite[Lemma~3]{GG}}]\label{lem:old}
Let $B$ be an RB-operator on a finite group $G$ such that 

(i) $R = \Imm(B)\cap \Imm(\widetilde{B})$ is abelian,

(ii) $G = \ker(B)\cdot \Imm(B)$.

Then $\widetilde{B}|_{\Imm(B)}$ is a~homomorphism onto abelian subgroup~$R$ of $\Imm(B)$,
and $B$ is defined on $G$ by Proposition~\ref{coro:old}.
\end{lemma}

\begin{lemma}\label{lem:R=2}
Let $G= HK$, $H_1\unlhd H$, $K_1\unlhd K$ and $H/H_1\simeq K/K_1$. Let $G=H_1K$ be an exact factorization of $G$ and $R=H\cap K$ be a subgroup of order two, which normalizes $H_1$. Let $t\in K\setminus K_1$ and $r\in R\setminus\{e\}$. Every element $g\in G$ has the form $g =h_1k=h_1t^\delta k_1$, where $h_1\in H_1, k\in K,k_1\in K_1$ and $\delta\in\{0,1\}$. 
Then a map $B\colon G\rightarrow G$ defined as
\begin{equation}\label{eq:ActionBWhen|R|=2}
B(g)
 = B(h_1k)
 = B(h_1t^\delta k_1)
 = k^{-1}r^\delta
\end{equation}
is a Rota--Baxter operator.
\end{lemma}
\begin{proof}
A map $C=B|_K\colon G\rightarrow G$ acts as 
$C(k) = k^{-1}r^\delta$. 
Hence, $\widetilde{C}(k)=k^{-1}C(k^{-1})=r^\delta$ is a homomorphism onto
$R\simeq\mathbb{Z}_2$ and therefore $C$ and $\widetilde{C}$ are Rota--Baxter operators on~$K$. 
Since $\Imm(\widetilde{C})=R$ normalizes $H_1$, $B$ is a Rota--Baxter operator by Proposition~\ref{coro:old}.
\end{proof}

Now we apply this lemma to give an example of a simple classical group with a~non-splitting Rota--Baxter operator.

\begin{example}\label{Ex:PSp_6(2)}
According to~\cite[Table~10.7]{LP_fact}, the group $G=\PSp_6(2)$ has the following factorization
$$G=HK=\Imm B\cdot\Imm \widetilde{B}=(\SU_3(3)\!:\!2)\cdot(\Ss_5\times2).$$ 
In addition, $R=H\cap K=\mathbb{Z}_2$ and $\PSp_6(2)=\SU_3(3)\cdot(\Ss_5\times2)=(\SU_3(3)\!:\!2)\cdot\Ss_5$ are exact factorizations of $G$. We are in the situation of Lemma~\ref{lem:R=2}, which provides a non-splitting Rota--Baxter operator on $G$. 
\end{example}

\begin{lemma}\label{Extension}
Let $G=\langle A,f\rangle$, where $A$ is a normal abelian subgroup of $G$. Let
$B \colon G\rightarrow G$ be a map such that
$B|_{A}$ is a homomorphism, $B(f)\in A$, and $B(f^ka)=B(f)^kB(a)$ for $k\in\mathbb{Z}$. Then $B$ is a~Rota--Baxter operator if and only if $[\Imm(\widetilde{B}),f]\leqslant \ker(B)$.
\end{lemma}

\begin{proof}
We need to check the equality $$B(g)B(h)=B(gB(g)hB(g)^{-1})$$
for all $g,h \in G$. From the construction of $B$ it follows that $\Imm(B)\subseteq A$. Since $A$ is abelian, the required equality holds for $g\in G$ and $h\in A$.

Let $g=xa, h=yb$, where $a,b\in A$ and $x=f^m$, $y=f^n$ for some integers $m,n$. Then
\begin{multline*}
B(gB(g)hB(g)^{-1})=B(xaB(xa)ybB(xa)^{-1})=B(xy(aB(xa))^ybB(xa)^{-1}) \\
 = B(xy)\cdot B((aB(xa))^y)\cdot B(b)\cdot B(B(xa)^{-1}).
\end{multline*}
On the other hand,
$$
B(g)B(h)=B(xa)B(yb)=B(x)B(a)B(y)B(b)=B(xy)B(a)B(b).
$$
Hence, our equality is equivalent to
$$
B(a)=B((aB(xa))^y)\cdot B(B(xa)^{-1}),
$$
and
$$
1=B((aB(xa))^y)\cdot B(B(xa)^{-1})\cdot B(a^{-1})=B((aB(xa))^yB(xa)^{-1}a^{-1}).
$$
Denote $z=aB(xa)$. Since $z\in A$ and $[x,y]=1$, we have
\begin{equation*}
1=B(z^yz^{-1})=B(z^{-1}z^y)=B([z,y])=B([xz,y])=
B([xaB(xa),y])=B([\widetilde{B}((xa)^{-1}),y]).
\end{equation*}

It means that $[\widetilde{B}(t),y]\in \ker(B)$ for every $t\in G$ and $y\in\langle f \rangle$, which is equivalent to $[\Imm(\widetilde{B}),f]\leqslant\ker(B)$.
\end{proof}

Now we apply Lemma~\ref{Extension} to construct a non-splitting Rota--Baxter operator on a~group of order $16$. This is the first RB-operator that does not correspond to an exact factorization or to a homomorphism into an abelian subgroup.

\begin{example}\label{ex:extraspecial_group}
 Consider the group 
 $$G=\langle a,b,c\, |\, a^4=b^2=c^2=1, [a,b]=[c,b]=1, a^c=ab \rangle\simeq(\mathbb{Z}_4\times\mathbb{Z}_2)\rtimes\mathbb{Z}_2.$$
 Then $A=\langle a^2,b,c\rangle\simeq\mathbb{Z}_2\times\mathbb{Z}_2\times\mathbb{Z}_2$ is a normal abelian subgroup of $G$. Define $B \colon G\rightarrow G$ as follows: 
 $$
 B|_{A}  \text{ is a homomorphism with } B(a^2)=1, B(b)=a^2b, B(c)=a^2bc,
 $$ 
 $$B(a)=a^2\in A, \quad B(ax)=B(a)B(x), \quad x\in A.$$

 It is straightforward to check that $\Imm(\widetilde{B})=\langle a, b\rangle$. Hence, $[\Imm(\widetilde{B}),a]=1$ and $B$ is a~Rota--Baxter operator on $G$ by Lemma~\ref{Extension}.
\end{example}

\section{Proof of the main results}

{\it Proof of Theorem~\ref{th:exceptional}}. Let $B$ be a Rota--Baxter operator on a finite simple exceptional group of Lie type $G$. By Proposition~\ref{Prop:STSh}(d) we have factorization $G = \Imm(B)\Imm(\widetilde{B})$. All factorizations of $G$ were described in~\cite{HLS} and are presented in Table~\ref{table:factorization} (see \cite[Table~5]{LPS_fact}). By Proposition~\ref{Prop:simple}(1) we need to consider only the groups that have non-trivial factorizations.

Assume that $G=\operatorname{F_4(q)}$, $q=2^m$. The only two possible factorizations of $G$ are $G=HL$ with $H=\operatorname{Sp_8(q)}$ and $L\in\{\operatorname{^3D_4(q)}, \operatorname{^3D_4(q).3}\}$. Since $\operatorname{Sp_8(q)}$ is simple, $B$ cannot be non-splitting by Proposition~\ref{Prop:simple}(2). 
On the other hand, these two factorizations are not exact and cannot lead to a splitting Rota--Baxter operator. Thus, we have only trivial factorizations of $G$ that give trivial Rota--Baxter operators.

Assume that $G=\operatorname{G_2(q)}$. Consider the possible factorizations of $G$ \cite[Table~5]{LPS_fact}. The factorizations 
$$
\operatorname{G_2(4)}=\operatorname{J_2}\cdot\operatorname{SU_3(4)}=\operatorname{J_2}\cdot\operatorname{SU_3(4)}.2
$$
are not exact and have a simple factor $\operatorname{J_2}$.
Thus, these factorizations do not correspond to a Rota--Baxter operator. 

Let $q=3^{2m+1}$ and $$\operatorname{G_2(q)}=\operatorname{^2G_2(q)}\cdot\operatorname{SL_3(q)}=\operatorname{^2G_2(q)}\cdot\operatorname{SL_3(q)}.2.$$
Direct calculations show that
$|\operatorname{^2G_2(q)}\cap\operatorname{SL_3(q)}|=q-1$. 
Hence, these factorizations are not exact and have a simple factor $\operatorname{^2G_2(q)}$. 
Thus, these factorizations do not correspond to a Rota--Baxter operator. 

Finally, let $q=3^m$ and
$$
\operatorname{G_2(q)}=HL, \text{ where } 
H\in\{\operatorname{SL_3(q)},\operatorname{SL_3(q)}.2\} \text{ and } L\in\{\operatorname{SU_3(q)},\operatorname{SU_3(q)}.2\}.
$$
Denote $R=H\cap L$. 
Then $|R|=k\cdot|\operatorname{SL_3(q)}\cap \operatorname{SU_3(q)}|=k(q^2-1)$, where $k\,|\,4$. 
In particular, these factorizations are not exact. 
We also have $|R|=|\Imm(B):\ker(\widetilde{B})|$ (\cite[equation (8)]{GG}) and $\ker(\widetilde{B})\unlhd\Imm(B)$. On the other hand, every proper normal subgroup of $\operatorname{SL_3(q)}$ lies in the center $\operatorname{Z(SL_3(q))}$. 
Since $q=3^m$, we have $|\operatorname{Z(SL_3(q))}|=(3,q-1)=1$. Hence, the case $\Imm(B)=\operatorname{SL_3(q)}$ implies $\ker(\widetilde{B})=\Imm(B)$ or $\ker(\widetilde{B})=1$.
If $\Imm(B) = \operatorname{SL_3(q).2}$ and $\ker(\widetilde{B}) =\operatorname{SL_3(q)}$, then we have $|R|=2<k(q^2-1)$, a contradiction. 
Thus, $B$ must be a trivial Rota--Baxter operator.

\begin{table}
	\caption{Factorizations of exceptional groups}\label{table:factorization}
        \smallskip
	\centering
	\begin{tabular}{|lll|}
		\hline
		$G$ & $H$ & $L$ \\ \hline
		$\operatorname{G_2(q)}$, $q=3^m$ & $\operatorname{SL_3(q)}$ or $\operatorname{SL_3(q)}\!.2$  & $\operatorname{SU_3(q)}$ or $\operatorname{SU_3(q)}\!.2$  \\ \hline
        $\operatorname{G_2(q)}$, $q=3^{2m+1}$ & $\operatorname{SL_3(q)}$ or $\operatorname{SL_3(q)}\!.2$  & $\operatorname{^2G_2(q)}$  \\ \hline
        $\operatorname{F_4(q)}$, $q=2^m$ & $\operatorname{Sp_8(q)}$  & $\operatorname{^3D_4(q)}$ or $\operatorname{^3D_4(q)}\!.3$  \\ \hline
        $\operatorname{G_2(4)}$ & $\operatorname{J_2}$  & $\operatorname{SU_3(4)}$ or $\operatorname{SU_3(4)}\!.2$  \\ \hline
	\end{tabular}
\end{table}

\vspace{1em}

\noindent{\it Proof of Theorem~\ref{th:PSL2}}. Let $B$ be a Rota--Baxter operator on $G=\operatorname{PSL_2(q)}$. By Proposition~\ref{Prop:STSh}(d) we have factorization $G = \Imm(B)\Imm(\widetilde{B})$. All maximal factorizations of $G$ were described in~\cite[Tables~1,~3]{LPS_fact}. All factorizations $X=AB$, where $X$ is a finite almost simple group and $A,B$ are core-free subgroups such that $A\cap B$ is cyclic or dihedral, were determined in~\cite{LP_fact}. In particular, all exact factorizations can be found in~\cite[Tables~10.3, 10.4]{LP_fact}.

At first, consider the case $q=2^m\geqslant4$. 
Then we have only one maximal factorization $G=HL$, where $H=\mathbb{Z}_2^m\!:\!\mathbb{Z}_{q-1}$ and $L=\operatorname{D_{2(q+1)}}$. Notice that $H$ has no subgroups of index~2.
Denote $R=H\cap L$. Then $|R|=2$, and this factorization is not exact. If $H=\Imm(B)$ and $L=\Imm(\widetilde{B})$, we have $2=|R|=|\Imm(B):\ker(\widetilde{B})|$ and $H$ has a subgroup of index 2, which is not true.
Hence, the only possible factorization is $G=HL'$, where $L'=\operatorname{\mathbb{Z}_{q+1}}$. This factorization is exact and gives the splitting Rota--Baxter operator on~$G$.

Now consider the case $q=p^m\geqslant5$ with an odd prime $p$. 
Let $H=\mathbb{Z}_p^m\!:\!\mathbb{Z}_{\frac{q-1}{2}}$ and $L=\operatorname{D_{q+1}}$ $(q\neq7,9)$ be the maximal subgroups of $G$. According to~\cite[Table~1]{LPS_fact}, we have a factorization $G=HL$ if and only if $q\equiv3\!\pmod4$. This factorization is exact, and we have a splitting Rota--Baxter operator. 

The other maximal factorizations can be found in~\cite[Table~3]{LPS_fact}. All these maximal factorizations are also presented in~\cite[Table~10.4]{LP_fact}, and therefore all possible factorizations of these groups are presented in~\cite[Table~10.4]{LP_fact}. All exact factorizations give us splitting Rota--Baxter operators (see Table~\ref{table}). Since $\operatorname{A_5}$ is simple, all non-exact factorizations with $\operatorname{A_5}$ do not lead to a Rota--Baxter operator.

The last factorization is $\operatorname{PSL_2(7)}=HL=\operatorname{S_4}\cdot 7.3$. We have $|R|=|H\cap L|=3$. Since $\operatorname{S_4}$ does not have a normal subgroup of index $3$, this factorization does not lead to a~Rota--Baxter operator.

All non-trivial Rota--Baxter operators on $\operatorname{PSL_2(q)}$ are presented in Table~\ref{table}. Now we explain why every row of Table~\ref{table} corresponds to exactly one splitting Rota--Baxter operator up to the equivalence.

Let $G=HL=H'L'$ be two exact factorizations from a row of Table~\ref{table}. 
Notice that there are two conjugacy classes of the maximal subgroups of $\operatorname{A_5}$ and $\operatorname{S_4}$ in the corresponding groups and one conjugacy class of the other factors. These two conjugacy classes of $\operatorname{A_5}$ or $\operatorname{S_4}$ are conjugate in the automorphism group by some outer automorphism~$\varphi$. 
Hence, there exist $\varphi\in\Aut(G)$ and $\alpha_x\in\Inn(G)$ such that $H^\varphi=H'$ and $(L^{\alpha_x})^{\varphi}=L'$. Putting $\Phi=(\varphi,\varphi\alpha_x)$, we have $\Phi(H_B)=H_{B'}$, which means that our factorizations correspond to equivalent Rota--Baxter operators.
\hfill $\square$

\begin{table}
	\caption{Rota--Baxter operators on $\operatorname{PSL_2(q)}$}\label{table}
        \smallskip
	\centering
	\begin{tabular}{|llll|}
		\hline
		Group & $\Imm(B)$ & $\Imm(\widetilde{B})$ & Remark\\ \hline
$\operatorname{PSL_2(p^m)}$ & $\operatorname{D_{q+1}}$  & $p^m\!:\!(\frac{q-1}{2})$ &  $q\equiv3\!\pmod4$, $p\neq7,11,23,59$  \\ \hline
$\operatorname{PSL_2(p^m)}$ & $\operatorname{\mathbb{Z}_{q+1}}$  & $p^m\!:\!(q-1)$ & $p=2$ \\ \hline
        $\operatorname{PSL_2(7)}$ & $\operatorname{S_4}$  & 7 &  \\ 
        $\operatorname{PSL_2(7)}$ & $\operatorname{D_8}$  & 7\,:\,3 &  \\ \hline
        $\operatorname{PSL_2(11)}$ & $\operatorname{A_5}$  & 11 & \\
        $\operatorname{PSL_2(11)}$ & $\operatorname{A_4}$  & 11\,:\,5  & \\
        $\operatorname{PSL_2(11)}$ & $\operatorname{D_{12}}$  & 11\,:\,5 &  \\ \hline
        $\operatorname{PSL_2(23)}$ & $\operatorname{S_4}$  & 23\,:\,11 &  \\
        $\operatorname{PSL_2(23)}$ & $\operatorname{D_{24}}$  & 23\,:\,11 & \\ \hline
        $\operatorname{PSL_2(59)}$ & $\operatorname{A_5}$  & 59\,:\,29  & \\
        $\operatorname{PSL_2(59)}$ & $\operatorname{D_{60}}$  & 59\,:\,29  & \\ \hline
	\end{tabular}
\end{table}

Note that the equivalence on the set of RB-operators is stronger than a conjugation by the automorphism group. 
In particular, two conjugacy classes of $\operatorname{S_4}$ and $\operatorname{A_5}$ mentioned above provide two RB-operators that are not conjugate in the automorphism group of $\operatorname{PSL_2(23)}$ and $\operatorname{PSL_2(59)}$ respectively.

\section*{Acknowledgements}

The study was supported by a grant from the Russian Science Foundation \textnumero~23-71-10005, https://rscf.ru/project/23-71-10005/

\noindent Alexey Galt \\
Vsevolod Gubarev \\
Novosibirsk State University \\
Pirogova str. 1, 630090 Novosibirsk, Russia \\
Sobolev Institute of Mathematics \\
Acad. Koptyug ave. 4, 630090 Novosibirsk, Russia \\
e-mail: galt84@gmail.com, wsewolod89@gmail.com


\begin{thebibliography}{99}

\bibitem{Aguiar00}
M. Aguiar,
Pre-Poisson algebras, Lett. Math. Phys. {\bf 54} (2000) 263--277.

\bibitem{BBGN2011}
C. Bai, O. Bellier, L. Guo, X. Ni,
Splitting of operations, Manin products, and Rota--Baxter operators,
Int. Math. Res. Notices {\bf 3} (2013) 485--524.

\bibitem{BG}
V.G. Bardakov, V. Gubarev, 
Rota--Baxter operators on groups, Proc. Indian Acad. Sci. (Math. Sci.), (1) {\bf 133} (2023), N4.

\bibitem{BG2}
V.G. Bardakov, V. Gubarev, 
Rota--Baxter groups, skew left braces, and the Yang-Baxter equation, J. Algebra {\bf 596} (2022), 328--351.

\bibitem{Baxter}
G. Baxter,
An analytic problem whose solution follows from a simple algebraic identity,
Pacific J. Math. {\bf 10} (1960) 731--742.

\bibitem{BelaDrin82}
A.A. Belavin, V.G.  Drinfel'd.
Solutions of the classical Yang--Baxter equation for simple Lie algebras,
Funct. Anal. Appl. (3) {\bf 16} (1982) 159--180.

\bibitem{Catino2023}
F. Catino, M. Mazzotta, P. Stefanelli, 
Rota--Baxter operators on Clifford semigroups and the Yang--Baxter equation, 
J. Algebra {\bf 622} (2023), 587--613.

\bibitem{Das}
A. Das, N. Rathee,
Extensions and automorphisms of Rota--Baxter groups,
J. Algebra {\bf 636} (2023), 626--665.

\bibitem{GG}
A. Galt and V. Gubarev, Rota--Baxter operators on dihedral and alternating groups, Advances in Group Theory and Applications {\bf 23}, 2026.

\bibitem{RBHopf}
M. Goncharov,
Rota--Baxter operators on cocommutative Hopf algebras, J. Algebra {\bf 582} (2021), 39--56.

\bibitem{DoubleLie2}
M. Goncharov, V. Gubarev,
Double Lie algebras of nonzero weight,
Adv. Math. {\bf 409} (2022), 108680.

\bibitem{DoubleLie}
M.E. Goncharov, P.S. Kolesnikov.
Simple finite-dimensional double algebras, J.~Algebra {\bf 500} (2018) 425--438.

\bibitem{GV2017}
L. Guarnieri and L. Vendramin,
Skew braces and the Yang--Baxter equation,
Math. Comp. {\bf 86} (2017), 2519--2534.

\bibitem{GuoMonograph}
L. Guo,
An Introduction to Rota--Baxter Algebra. Surveys of Modern Mathematics, vol. 4,
International Press, Somerville (MA, USA); Higher education press, Beijing, 2012.

\bibitem{Jiang}
J. Jiang, Y. Sheng and C. Zhu, Lie theory and cohomology of relative Rota-Baxter operators,
J. Lond. Math. Soc. (2) {\bf 109} (2024), e12863.

\bibitem{Jiang2}
J. Jiang, Y. Sheng and C. Zhu,
On the integration of relative Rota-Baxter Lie algebras, arXiv:2410.23547.

\bibitem{HLS}
C. Hering, M.W. Liebeck and J. Saxl, The factorizations of the finite exceptional groups of Lie type, J. Algebra {\bf 106} (1987), 517--527.

\bibitem{DoublePoissonFree}
A. Odesskii, V. Rubtsov, V. Sokolov,
Double Poisson brackets on free associative algebras,
in: Noncommutative birational geometry, representations and combinatorics,
Contemp. Math. {\bf 592} (2013) 225--239, AMS, Providence, RI.

\bibitem{LPS_fact}
M.W. Liebeck, C.E. Praeger, J. Saxl, 
The maximal factorizations of the finite simple groups and their automorphism groups, Mem. Am. Math. Soc. 432 (1990).

\bibitem{LP_fact}
M.W. Liebeck, C.E. Praeger, 
Maps, simple groups, and arc-transitive graphs, Advances in Mathematics, 462 (2025) 110086.

\bibitem{Guo2020}
L. Guo, H. Lang, Yu. Sheng,
Integration and geometrization of Rota--Baxter Lie algebras,
Adv. Math. {\bf 387}, 107834 (2021).

\bibitem{Semenov83}
M.A. Semenov-Tyan-Shanskii,
What is a classical $r$-matrix?
Funct. Anal. Appl. {\bf 17} (1983) 259--272.

\bibitem{Skresanov}
S.V. Skresanov, Rota–Baxter operators on compact simple Lie groups and algebras, arXiv:2506.14324.
\end{thebibliography}
\end{document}